\documentclass{article}

\newcommand{\overbar}[1]{\mkern 1.5mu\overline{\mkern-1.5mu#1\mkern-1.5mu}\mkern 1.5mu}

\newcommand{\sseq}{\subseteq}

\newcommand{\IFF}{\Longleftrightarrow}
\newcommand{\ith}{i^{\text{th}}}
\newcommand{\mdeg}{\mathrm{mdeg}}
\newcommand{\lcm}{\mathrm{lcm}}
\newcommand{\D}{\Delta}
\newcommand{\pd}{\mathrm{pd}}
\newcommand{\U}{\mathcal{U}}
\newcommand{\T}{\mathcal{T}}
\newcommand{\K}{\mathcal{K}}

\usepackage[all]{xy}
\usepackage{graphics}
\usepackage{epstopdf}
\usepackage{enumerate}
\usepackage{amssymb}
\usepackage{times}
\usepackage{etex}
\usepackage{amsmath,amsfonts,amsthm,amssymb,thmtools,mathtools}
\usepackage{setspace}
\usepackage{fancyhdr}
\usepackage{lastpage}
\usepackage{extramarks}
\usepackage{chngpage}
\usepackage{soul}
\usepackage[usenames,dvipsnames]{color}
\usepackage{ifthen}
\usepackage{listings}
\usepackage{courier}
\usepackage{url}
\usepackage{pifont}
\usepackage{longtable}
\usepackage{rotating}
\usepackage{float}

\usepackage{setspace}

\theoremstyle{plain}
\newtheorem{theorem}{Theorem}
\newtheorem{lemma}[theorem]{Lemma}
\newtheorem{proposition}[theorem]{Proposition}
\newtheorem{corollary}[theorem]{Corollary}
\theoremstyle{definition}
\newtheorem{definition}[theorem]{Definition}
\newtheorem{example}[theorem]{Example}

\newtheorem{remark}[theorem]{Remark}

\author{Sara Faridi\thanks{Department of Mathematics and Statistics,
    Dalhousie University, Halifax, NS, Canada, faridi@mathstat.dal.ca.
    Research supported by NSERC.}\ \  and \ Ben Hersey\thanks{Department of
    Mathematics and Statistics, Queen's University, Kingston, ON,
    Canada, bsh1@queensu.ca}}
\date{}

\title{\Large \sc Resolutions of Monomial Ideals of
  Projective Dimension 1}

\begin{document}
\maketitle

\begin{abstract} We show that a monomial ideal $I$ has projective dimension $\leq$ 1 if and only if the minimal free resolution of $S/I$ is supported on a graph that is a tree. This is done by constructing specific graphs which support the resolution of the $S/I$. We also provide a new characterization of quasi-trees, which we use to give a new proof to a result by Herzog, Hibi, and Zheng which characterizes monomial ideals of projective dimension 1 in terms of quasi-trees.
\end{abstract}



\section{Introduction}

A free resolution of an ideal $I \subset k[x_1,...,x_n]$ is a long exact sequence of free modules that represents the relations between the generators of that ideal. If $I$ is generated by monomials it is always possible to find a simplicial complex whose simplicial chain complex determines a free resolution of the $S/I$ (e.g. the Taylor complex, see \cite{Bayer-MonRes}, \cite{Taylor-PhD}). However, it is known that the \emph{minimal} free resolution may not be determined by a simplicial complex; see Velasco~\cite{Velasco-MinResNotCW} and Reiner and Welker~\cite{Reiner-LinSyz}. In other words, the minimal resolution of $S/I$ need not be supported on a simplicial complex. A natural question is which monomial ideals do have such minimal simplicial resolutions?

In this paper we consider squarefree monomial ideals of projective dimension 1. The restriction to squarefree monomial ideals can be made without loss of generality since the polarization, $I_{\mathrm{pol}}$, of $I$ is a squarefree monomial idea for any monomial ideal $I$. Moreover, Peeva and Velasco \cite{Peeva-FramesDegen} show that the minimal free resolution of both $S/I$ and the minimal free resolution of $S_{\mathrm{pol}}/I_{\mathrm{pol}}$ are homogenizations of the same frame. In particular, the minimal free resolution of $S/I$ is supported on a simplicial complex if and only if $S_{\mathrm{pol}}/I_{\mathrm{pol}}$ is supported on the same simplicial complex. 

We show that if $I$ is a squarefree monomial ideal of projective dimension 1, then we can give a complete combinatorial characterization of both $I$ and the minimal free resolution of $S/I$. Our treatment of of these ideals and their resolutions allow us to prove the following theorem, in which we let $\mathcal{N}(I^\vee)$ denote the Alexander Dual of the Stanley-Reisner complex of $I$.

 
\begin{theorem}[Theorem~\ref{Cor to HHZ}] Let $I$ be a squarefree monomial ideal in a polynomial ring $S$. Then the following statements are equivalent.
  \begin{enumerate}  
  \item $\pd_S(I) \leq 1$
    \item $\mathcal{N}(I^\vee)$ is a quasi-forest
    \item $S/I$ has a minimal free resolution supported on a graph-tree.
  \end{enumerate}
\end{theorem}

The equivalence of (1) and (2) was established by Herzog, Hibi, and Zheng \cite{Herzog-Dirac} by applying the Hilbert-Burch theorem; it is a characterization of the minimal free resolution of $R/I$, when $R$ is a local ring (in our case a graded local ring) and $\mathrm{pd}_R(R/I) = 2$ (i.e. $\mathrm{pd}_R(I) = 1$), by the determinants of minors of the differential. The approach taken in this paper is to first establish the equivalence of (1) and (3) for a general monomial ideal $I$, and then show that when $I$ is additionally squarefree (2) holds if and only if (3) holds as well. Moreover, when $\mathcal{N}(I^\vee)$ is a quasi-forest we give an algorithm for constructing a labeled tree which supports the minimal free resolution of $I$.

This paper is split into two parts. The first will contain the necessary definitions and preliminary results that we need, with the only new result being a characterization of quasi-forests in terms of their induced subcomplexes. The second part is where we restrict to monomial ideals of small projective dimension and establish the equivalences in Theorem \ref{Cor to HHZ}.

\section{Preliminaries: Quasi-Forests and Simplicial Resolutions}

\begin{definition}  Let $V= \{v_1,...,v_n\}$ be a finite set. A (finite) \textbf{simplicial complex}, $\Delta$, on $V$ is a collection of non-empty subsets of $V$ such that $F \in \Delta$ whenever $F \subseteq G$ for some $G \in \Delta$. The elements of $\Delta$ are called \textbf{faces}. Faces containing one element are called \textbf{vertices} and the set $V(\Delta) = \{ v_i \ | \ \{v_i\} \in \Delta \}$ is called the \textbf{vertex set} of $\Delta$. The maximal faces of $\Delta$ are called \textbf{facets}. For each face $F \in \Delta$, we define $\dim(F) = |F|-1$ to be the \textbf{dimension} of $F$. We define $\dim(\Delta) = \max\{\dim(F) : F \in \Delta\}$ to be the dimension of the simplicial complex $\Delta$. If $\Delta$ is a simplicial complex with only 1 facet and $r$ vertices, we call $\Delta$ an \textbf{r-simplex}.
\end{definition}
  
\begin{definition} If $W \sseq V$, we define the \textbf{induced subcomplex} of $\Delta$ on $W$, denoted $\Delta_W$, to be the simplicial complex on $W$ given by $\Delta_W = \{F \in \Delta | F \sseq W\}$. A \textbf{subcollection} of $\Delta$ is a simplicial complex whose facets are also facets of $\Delta$. We say $\Delta$ is \textbf{connected} if for every $v_i,v_j \in V$ there is a sequence of faces $F_0,...,F_k$ such that $v_i \in F_0$, $v_j \in F_k$ and $F_i \cap F_{i+1} \not = \emptyset$ for $i = 0,...,k-1$.
\end{definition}
  
It is easy to see from the definition that a simplicial complex can be described completely by its facets, since every face is a subset of a facet and every subset of every facet is in a simplicial complex. So, if $\Delta$ has  facets $F_0,...,F_q$, we use the notation $\langle F_0,...,F_q \rangle$ to describe $\Delta$.

The \textbf{$f$-vector} of a $d$-dimensional simplicial complex $\Delta$ is the sequence $f(\Delta) = (f_0,...,f_d)$, where  each $f_i$ is the number of $i$-dimensional faces of $\Delta$.

\begin{definition}[Faridi~\cite{Faridi-FacetIdeal}] A facet $F$ of a simplicial complex $\Delta$ is called a \textbf{leaf} if either $F$ is the only facet of $\Delta$ or for some facet $G \in \Delta$ with $G\neq F$ we have that $F \cap H \sseq G$ for all facets $H\neq F$ of $\Delta$. The facet $G$ is said to be the \textbf{joint} of $F$. A simplicial complex $\Delta$ is a \textbf{simplicial forest} if every nonempty subcollection of $\Delta$ has a leaf. A connected simplicial forest is called a  \textbf{simplicial tree}.

\end{definition}

If a facet $F$ of a simplicial complex is a leaf, then $F$ necessarily has a \textbf{free vertex}, which is a vertex of $\Delta$ that belongs to exactly one facet.

One of the properties of simplicial trees that we will make particular use of is that whenever $\Delta$ is a simplicial tree we can always order the facets $F_1,...,F_q$ of $\Delta$ so that $F_i$ is a leaf of the induced subcollection $\langle F_1,...F_i \rangle$. Such an ordering on the facets is called a \textbf{leaf order} and it is used to make the following definition.

\begin{definition}(Zheng~\cite{Zheng-ResFac}) \label{quasi-tree} A simplicial complex $\Delta$ is a \textbf{quasi-forest} if $\Delta$ has a leaf order. A connected quasi-forest is called a {\bf quasi-tree}.
\end{definition}

Equivalently, we could have defined quasi-trees to be simplicial complexes such that every induced subcomplex has a leaf. This is not clear from the definition, so we give a proof.

\begin{proposition}[A characterization of quasi-forests]\label{prop 2} A simplicial complex $\Delta$ with vertex set $V$ is a quasi-forest if and only if for every subset $W \subset V$, the induced subcomplex $\Delta_W$ has a leaf.
\end{proposition}
\begin{proof}
($\Rightarrow$) Since $\Delta$ has a leaf order, we may label the facets of $\Delta, \ F_0,...,F_q$, so that $F_i$ is a leaf of $\Delta_i = \langle F_0,...,F_i \rangle$. For a subset $W \subset V$, choose the smallest $i$ such that $W$ is a subset of the vertex set of $\Delta_i$.

We claim that the complex induced on $W$ in $\Delta_i$ is $\Delta_W$. It is clear that $(\Delta_i)_W \subseteq \Delta_W$. To see the converse, let $F$ be a face of $\Delta_W$, then $F \subseteq F_j$ for some facet $F_j \in \Delta$. If $j \leq i$ then $F \in \Delta_i$ and we are done. If $j > i$ then let $F_k$ be the joint of $F_j$ in $\Delta_j$ and note that $k < j$. Since $F \in \Delta_W \subseteq \Delta_i \subseteq \Delta_j \setminus \langle F_j \rangle$ we have that $F \subseteq F_j \cap \big( \Delta_j \setminus \langle F_j \rangle \big) \subset F_k$. If $k \leq i$ then we are done. If not we may iterate this argument as many times as necessary until we get a facet $F_a \in \Delta_i$ for which $F \subseteq F_a$. Hence $(\Delta_i)_W = \Delta_W$.

We will show that $F_i \cap W$ is a leaf of $\Delta_W$. Since $F_i \in \Delta_i, \ F_i \cap W$ is a face of $\Delta_W$. Let $V_i$ be the vertex set of $\Delta_i$, then we also have that $V_i = V_{i-1} \cup \{\text{free vertices of }F_i \ \text{in} \ \Delta_i\}$ which means that $W \cap \{\text{free vertices of }F_i \ \text{in} \ \Delta_i\} \not = \emptyset$, otherwise $W$ would be contained in the vertex set of $\Delta_{i-1}$. Therefore $F_i \cap W$ is not a subset of any other face in $\Delta_W$, i.e. $F_i \cap W$ is a facet of $\Delta_W$. If $F_j$ is the joint of $F_i$ in $\Delta_i$, then for any face $F \in \Delta, \ F \cap F_i \cap W \subset F_j \cap F_i \cap W$. This means that any facet of $\Delta_W$ (except for $F_i \cap W$) that contains $F_j \cap F_i \cap W$ is a joint for $F_i \cap W$ in $\Delta_W$, since the faces of $\Delta_W$ are also faces of $\Delta$. If no such facet exist (except for $F_i \cap W$) then $F_i \cap W$ is disjoint from the rest of $\Delta_W$. In either scenario, $F_i \cap W$ is a leaf of $\Delta_W$.

($\Leftarrow$) This is done by induction on the size of the vertex set $V$ of $\Delta$. For $|V|$ = 1 or 2, a quick inspection shows that all simplicial complexes with vertex set $V$ have a leaf order and every induced subcomplex has a leaf. Now assume that every simplicial complex on $\leq n$ vertices for which every induced subcomplex has a leaf is a quasi-forest.

Suppose $\Delta$ is a simplicial complex on $n+1$ vertices and that every induced subcomplex of $\Delta$ has a leaf. Since $\Delta$ is an induced subcomplex of itself, it also has a leaf, call it $F$, with free vertices $v_1,...,v_k$. The simplicial complex $\Delta \setminus \langle F \rangle$ is given by the induced subcomplex $\Delta_W$ where $W = V \setminus \{v_1,...,v_k\}$. Every induced subcomplex of $\Delta_W$ has a leaf and $\Delta_W$ is a simplicial complex on $\leq n$ vertices, hence $\Delta_W$ has a leaf order $G_1,....,G_j$. This gives us a leaf order $G_1,....,G_j,F$ for $\Delta$.
\end{proof}

It is known that every induced subcomplex of a simplicial forest is also a simplicial forest (\cite{Faridi-MonRes}), but this property does not characterize simplicial forests.

\begin{definition}
  Let $\Delta$ be a simplicial complex on $V = \{x_1,...,x_n\}$. The \textbf{Stanley-Reisner ideal} of $\Delta$ is a squarefree monomial ideal, $\mathcal{N}(\Delta) \subseteq k[x_1,...,x_n]$, generated by the minimal ``non-faces'' of $\Delta$:
\begin{displaymath}
  \mathcal{N}(\Delta) = (x_{i_1}\cdots x_{i_p} | \{x_{i_1},...,x_{i_p}
  \} \not \in \Delta).
\end{displaymath}
Conversely, let $I \subseteq k[x_1,...,x_n]$ be a squarefree monomial ideal. The \textbf{Stanley-Reisner complex} of $I$ is the simplicial complex $\mathcal{N}(I)$  on $V$ given by
\begin{displaymath}
\mathcal{N}(I) = \big \{ \{x_{i_1},...,x_{i_p}\} \ | \ x_{i_1} \cdots x_{i_p} \not \in I \big \}
\end{displaymath}
\end{definition}
\begin{definition}
  Let $\Delta$ be a simplicial complex on $V = \{x_1,...,x_n\}$. The \textbf{Alexander dual} of $\Delta$ is the simplicial complex
\begin{displaymath}
  \Delta^{\vee} = \{ \{ x_1,...,x_r \} \setminus \tau \ | \ \tau \not \in \Delta \}.
\end{displaymath}
For a squarefree monomial ideal $I \subset k[x_1,...,x_n]$, we define the \textbf{Alexander dual} of $I$ as
\begin{displaymath}
  I^\vee = \mathcal{N}\big((\mathcal{N}(I))^\vee\big)
\end{displaymath}
\end{definition}
The Stanley-Reisner operator $\mathcal{N}$ gives a bijective correspondence between simplicial complexes on $V = \{x_1,...,x_n\}$ and squarefree monomial ideals in $k[x_1,...,x_n]$. Moreover, we have that $\mathcal{N}(\mathcal{N}(\Delta)) = \Delta$ and $\mathcal{N}(\mathcal{N}(I)) = I$. Similarly, we have that $(\Delta^\vee)^\vee = \Delta$ and $(I^\vee)^\vee = I$ (\cite{Faridi-CohenMacaulay}).

These definitions give us tools for constructing and classifying squarefree monomial ideals via simplicial complexes and vice versa. In particular, we are interested in the constructions $\mathcal{N}(\Delta^\vee)$ and $\mathcal{N}(I^\vee)$, where $\Delta$ is a simplicial complex, and $I$ is a squarefree monomial ideal. 

\begin{lemma}[Faridi, \cite{Faridi-CohenMacaulay}] \label{Facet Generator Correspondence} Let $\Delta = \langle F_1,...,F_q \rangle$ be a simplicial complex on $V = \{x_1,...,x_n\}$. Then the minimal generating set of $\mathcal{N}(\Delta^\vee)$ is ${m_1,...,m_q}$ where 
\begin{displaymath}
m_i = \prod_{x_j \not \in F_i} x_j
\end{displaymath}
Similarly, let $I = (m_1,...,m_r) \subseteq k[x_1,...,x_n]$ be a squarefree monomial ideal. Then the facets of $\mathcal{N}(I^\vee)$ are $F_1,...,F_r$ where
\begin{displaymath}
  F_i = \big \{ x_j \ \bigm|  \ x_j \not | \  m_i \big \}
\end{displaymath}
\end{lemma}
\begin{remark} \label{Complex-Ideal Correspondence}
   It follows from lemma \ref{Facet Generator Correspondence} that, like the Stanley-Reisner operator $\mathcal{N}$, the operator $\mathcal{N}((-)^\vee)$ also gives a bijective correspondence between simplicial complexes on $V = \{x_1,...,x_n\}$ and squarefree monomial ideals in $k[x_1,...,x_n]$. 
  
\end{remark}
\begin{definition} Let $S=k[x_1,\ldots,x_n]$, where $k$ is a field. A \textbf{minimal multigraded free resolution} of a multigraded $S$-module $M$ is a chain complex of the form
\begin{equation*}
\begin{xy}
(-10,0)*+{\mathbf{F} :};
(0,0)*+{...}="P0";
(15,0)*+{F_2}="P1";
(30,0)*+{F_1}="P2";
(45,0)*+{F_0}="P3";
(60,0)*+{0}="P5";
{\ar^{\partial_1} "P2"; "P3"}%
{\ar^{\partial_3} "P0"; "P1"}%
{\ar^{\partial_2} "P1"; "P2"}%
{\ar "P3"; "P5"}%
\end{xy}
\end{equation*}
such that each $F_i$ is a multigraded free $S$-module, $\partial_i$ a degree zero homomorphism, $H_0(\mathbf{F}) \cong M$ a degree zero isomorphism , $H_i(\mathbf{F}) = 0$ for $i \geq 1$, and $\partial_{i+1}(F_{i+1}) \sseq \mathbf{m}F_i$ for all $i \geq 0$ where $\mathbf{m} = (x_1,...,x_n)$. Note that monomial ideals are multigraded $S$-modules with respect to the standard $\mathbb N^n$ grading of $S$ and admit a multigraded minimal free resolution.
\end{definition}

\begin{definition}
The $\ith$ \textbf{Betti number} of $M$ over $S$ is defined as $\beta_i^S(M) := \mathrm{rank}(F_i)$. Since $\mathbf{F}$ is multigraded, each free module $F_i$ is a direct sum of modules of the form $S(-m)$, where $S(-m)$ is the free $S$-module generated by a single element in multidegree $m$. We define the \textbf{multigraded Betti numbers} of $M$ by
\begin{displaymath}
\beta_{i,m}^S(M) = \text{number of summands in}\ F_i\ \text{of the form}\ S(-m)
\end{displaymath}
\noindent for each multidegree $m$.

It is known that the minimal free resolution of a finitely generated $S$-module $M$ is unique up to isomorphism, hence the Betti numbers are invariants of $M$.  We can therefore define the \textbf{projective dimension} of $M$ as:
\begin{displaymath}
\mathrm{pd}_S(M) = \max \{ i\ |\ \beta_i^S(M) \not = 0 \}
\end{displaymath}
\end{definition}

Now let $I$ be a monomial ideal in $S$ minimally generated by $m_1,\ldots,m_t$. If $\D$ is a simplicial complex on $t$ vertices, one can label each vertex of $\D$ with one of the generators $m_1,\ldots,m_t$ and each face with the least common multiple of the labels of its vertices. From this labeling we construct a complex of $S$-modules by $I$-homogenizing the augmented simplicial chain complex of $\Delta$ with coefficients in $k$ (for an explicit description of this construction see \cite{Bayer-MonRes,Peeva-FramesDegen}). For any monomial $m$, we denote by $\D_m$ the subcomplex of $\D$ induced on the vertices of $\D$ whose labels divide $m$.

\begin{theorem}[Faridi, \cite{Faridi-MonRes}]\label{t:main}
  Let $\D$ be a simplicial tree whose vertices are labeled by monomials $m_1,\ldots,m_t \in S$, and let $I = (m_1,\ldots,m_t)$ be the ideal in $S$ generated by the vertex labels. The simplicial chain complex ${\mathcal{C}}(\D) = {\mathcal{C}}(\D; S)$ is a free resolution of $S/I$ if and only if the induced subcomplex $\D_m$ is connected for every monomial $m$.
\end{theorem}

An example of a simplicial tree $\D$ ``supporting'' a free resolution of
a monomial ideal (that is, the simplicial chain complex of $\D$ being
a free resolution of the ideal) is the Taylor
resolution~\cite{Taylor-PhD}, in which case $\D$ is a simplex (one
facet).

Theorem~\ref{t:main} implies that the Betti vector of $I$ (that is,
the vector whose $i$-th entry is the $i$-th Betti number of $I$) is
bounded by the $f$-vector of a simplicial tree $\D$ that supports a resolution of it:
\begin{eqnarray*}
\beta(I)=(\beta_0(I),\ldots,\beta_q(I))\leq
(f_0(\D),\ldots,f_q(\D))=\mathbf{f}(\D).
\end{eqnarray*}

Equality holds if some extra conditions are satisfied:

\begin{theorem}[Bayer, Peeva, Sturmfels, \cite{Bayer-MonRes}]\label{t:BPS} With notation as in
Theorem~\ref{t:main}, ${\mathcal{C}}(\D)$ is a minimal free resolution of $S/I$ if and
only if $m_A \neq m_{A'}$ for every proper subface $A'$ of a face $A$ of $\D$.
\end{theorem}
\begin{example}\label{Example of Homogenization}
  Consider the ideal $I = (x_1x_3x_6, x_1x_4x_6, x_1x_2x_4, x_4x_5x_6) \subset k[x_1,...,x_6]$, and the labeled simplicial complex
\begin{figure}[H] \centering
 \includegraphics{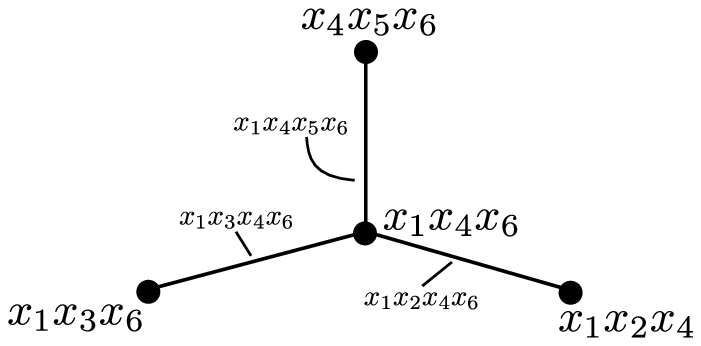}
\end{figure} 
\noindent From this labeled simplicial complex we construct the complex of $S$-modules
\begin{equation*}
\begin{xy}
(-20,0)*+{0}="P2";
(5,0)*+{\begin{smallmatrix}
    S(-x_1x_2x_4x_6)  \\
      \oplus \\
      S(-x_1x_3x_4x_6)  \\
      \oplus \\
      S(-x_1x_4x_5x_6)  \\
      \end{smallmatrix}}="P3";
(50,0)*+{\begin{smallmatrix}
      S(-x_1x_2x_4)  \\
      \oplus \\
      S(-x_1x_3x_6)  \\
      \oplus \\
      S(-x_1x_4x_6)  \\
      \oplus \\
      S(-x_4x_5x_6)
      \end{smallmatrix}}="P4";
(105,0)*+{S}="P5";
(120,0)*+{0}="P6";
{\ar "P2"; "P3"}%
{\ar^{\left[ \begin{smallmatrix}
       x_6   &   0    &  0       \\
       0     &   x_4  &  0       \\
      -x_2   &  -x_3  &  x_5      \\
       0     &   0    & -x_1 
      \end{smallmatrix}\right]} "P3"; "P4"}%
{\ar^(.55){\left[ \begin{smallmatrix}
      x_1x_2x_4 & x_1x_3x_6 & x_1x_4x_6 & x_4x_5x_6
      \end{smallmatrix}\right]} "P4"; "P5"}%
{\ar "P5"; "P6"}%
\end{xy}
\end{equation*}
which we can verify, through direct computation or by using theorems \ref{t:main} and \ref{t:BPS}, is the minimal multigraded free resolution of $S/I$. We will return to this example in the next section to show how to construct the tree supporting this minimal free resolution algorithmically using the quasi-tree $\mathcal{N}(I^\vee)$.  
\end{example}
\section{Monomial ideals of projective dimension 1}

It is known that not all monomial ideals have a simplicial, or even cellular, minimal free resolutions~(\cite{Reiner-LinSyz,Velasco-MinResNotCW}).  It is also known that if a simplicial complex supports a minimal resolution of a monomial ideal, then it must be acyclic~(\cite{Peeva-FramesDegen, Phan-MinMon}), and that simplicial trees are acyclic~(\cite{Faridi-MonRes}). A natural question to ask is: Which ideals have minimal resolutions supported on a simplicial tree?

We will address this questions for squarefree monomial ideals with projective dimension 1. Theorem \ref{pd is one iff graph tree} and Proposition \ref{resolution on graph tree is minimal} show that, unlike the general case, if a monomial ideal $I$ has projective dimension 1 then $I$ necessarily has a minimal free resolution supported on a simplicial complex which is, in fact, a tree in the context of both graphs and simplicial complexes. Theorem \ref{Quasi forest then resolution supported on a graph tree} gives a sufficient combinatorial criteria for determining if a squarefree monomial ideal has projective dimension 1. The remainder of the paper will then be devoted to establishing the necessity of these criteria, as well as providing some examples.

\begin{theorem} \label{pd is one iff graph tree}
A monomial ideal $I$ has $\pd(I) \leq 1$ if and only if $S/I$ has a minimal
resolution supported on a (graph) tree
\end{theorem}
\begin{proof}
($\Leftarrow$) Clear.

\noindent ($\Rightarrow$) We know that $\pd(I) = 0$ if and only if $I = (m)$ is principal, hence the minimal free resolution of $S/I$ is supported on the graph with a single vertex and no edges. Now, assume that $\pd(I) = 1$. Then $S/I$ has a minimal resolution
  of the form
\begin{equation*}
\begin{xy}
(-5,0)*+{0}="P0";
(10,0)*+{S^t}="P1";
(25,0)*+{S^r}="P2";
(40,0)*+{S}="P3";
(55,0)*+{0}="P4";
{\ar^{\phi} "P2"; "P3"}%
{\ar "P0"; "P1"}%
{\ar^{\psi} "P1"; "P2"}%
{\ar "P3"; "P4"}%
\end{xy}
\end{equation*}
where $\phi(e_i) = m_i$ for the basis elements $e_i$ of $S^r$, and
$\psi(g_j) = f_j$ where the $g_j$ form a basis of $S^t$ and the $f_j$
form a minimal generating set of $\ker(\phi)$. It is shown
(see~\cite{Ene-GrobComAlg}, Corollary 4.13) that $\ker(\phi)$ can be
generated (though not necessarily minimally) by the elements
\begin{displaymath}
\dfrac{\lcm(m_i,m_j)}{m_i} e_i -\dfrac{\lcm(m_i,m_j)}{m_j} e_j
\end{displaymath}

Let $f_1,...,f_t$ be a minimal generating set of $\ker(\phi)$ which have this form. This gives us a complete description of the map $\psi$ as a matrix with exactly two non-zero monomial entries in each column with coefficients corresponding to those appearing in the $f_i$ (i.e one column entry has coefficient $1$ and the other has coefficient $-1$). Dehomogenizing this resolution (i.e. tensoring the complex by $\dfrac{S}{(x_1-1,...,x_n-1)}$) gives us the sequence of vector spaces
\begin{equation} \label{Equation 1}
\begin{xy}
(-5,0)*+{0}="P0";
(10,0)*+{k^t}="P1";
(25,0)*+{k^r}="P2";
(40,0)*+{k}="P3";
(55,0)*+{0}="P4";
{\ar^{(11...1)} "P2"; "P3"}%
{\ar "P0"; "P1"}%
{\ar^{A} "P1"; "P2"}%
{\ar "P3"; "P4"}%
\end{xy}
\end{equation}
which is exact (Theorem 3.8 of~\cite{Peeva-FramesDegen}) and where $A$ is a matrix in which every column has exactly one entry which is 1, one entry which is -1, and the rest equal to zero. If we consider each basis element of $k^r$ as a vertex and each basis element $e_i$ of $k^t$ as an edge between the two vertices determined by the basis elements of $k^r$ to which $e_i$ is sent, we may construct a graph $G$ for which ${\mathcal{C}}(G;k)$ is the chain complex in (\ref{Equation 1}). Since this chain complex is exact the graph $G$ is acyclic, hence a tree (this would also imply that $t = r-1$).
\end{proof}

In fact, more is true. 

\begin{proposition}\label{resolution on graph tree is minimal}
If $I$ is a monomial ideal such that $S/I$ has a resolution supported on a tree $T$, then that resolution is minimal.
\end{proposition}
\begin{proof}
If $m_1,...,m_r$ are the minimal generators of $I$ then $T$ would have to have $r$ vertices and $r-1$ edges. When we regard $T$ as a simplicial complex we get the simplicial chain complex
\begin{equation*}
\begin{xy}
(-15,0)*+{\mathcal{C}(T;k) : 0}="P0";
(10,0)*+{k^{r-1}}="P1";
(30,0)*+{k^r}="P2";
(45,0)*+{k}="P3";
(60,0)*+{0}="P4";
{\ar^{(11...1)} "P2"; "P3"}%
{\ar "P0"; "P1"}%
{\ar^{\partial_2} "P1"; "P2"}%
{\ar "P3"; "P4"}%
\end{xy}
\end{equation*}
where $\partial_2$ is a matrix in which every column has one entry
equal to 1, one entry equal to $-1$, and the rest equal to zero. Fix a
basis $u_{i,j}$ for $\mathcal{C} (T;k)$. The $I$-homogenization of
$T$~(\cite{Peeva-FramesDegen}) would then give a resolution of $I$ of
the form
\begin{equation*}
\begin{xy}
(-15,0)*+{\mathbf{G}: \ 0}="P6";
(12,0)*+{\displaystyle{\bigoplus_{j = 1}^{r-1}} S(-\alpha_{2,j})}="P7";
(46,0)*+{\displaystyle{\bigoplus_{j = 1}^{r}} S(-\alpha_{1,j})}="P2";
(70,0)*+{S}="P3";
(80,0)*+{0}="P5";
{\ar "P6"; "P7"}%
{\ar^(.67){d_{1}} "P2"; "P3"}%
{\ar^{d_{2}} "P7"; "P2"}%
{\ar "P3"; "P5"}%
\end{xy}
\end{equation*}
with multihomogeneous basis $e_{i,j}$ such that $\mdeg (e_{i,j}) = \alpha_{i,j}$. We know that 
\begin{displaymath}
\alpha_{1,j} = \mdeg (e_{1,j}) = \mdeg (m_j)
\end{displaymath}
for $j = 1,...,r$ and the $\alpha_{2,j}$ are given by
\begin{displaymath}
\alpha_{2,j} = \mdeg \big( \lcm(\mdeg(e_{1,s}) |\  a_{s,j} \not = 0) \big)
\end{displaymath} 
where the $a_{s,j}$ come from the boundary map 
 \begin{displaymath}
\partial_2(u_{2,j}) = \sum_{s=1}^q a_{s,j} u_{1,s} 
\end{displaymath} 

For each $j$, exactly 2 of the $a_{s,j} \not = 0$, so the multidegrees of the $e_{2,j}$ are actually of the form $\mdeg(e_{2,j}) = \mdeg(\lcm(m_{i_1},m_{i_2}))$ where $m_{i_1}$ and $m_{i_2}$ are minimal generators of $I$. With this in mind we consider the boundary map
\begin{equation*}
d_2(e_{2,j}) = \sum_{s=1}^q a_{s,j} \frac{\mdeg(e_{2,j})}{\mdeg(e_{1,s})} e_{1,s}
\end{equation*}
which tells us that the matrix representation of $d_2$ has entries 
\begin{equation*}
[d_2]_{s,j} = a_{s,j} \frac{\mdeg (e_{2,j})} {\mdeg (e_{1,s})}
\end{equation*}

If $a_{s,j} = 0$ then $[d_2]_{s,j} = 0$. If $a_{s_1,j},a_{s_2,j} \not
= 0$ then we have that $\mdeg(e_{2,j}) = \lcm(m_{s_1},m_{s_2})$. Since
$m_{s_1},m_{s_2}$ are minimal generators of $I$ we know that $m_{s_1}$
and $m_{s_2}$ strictly divide $\mdeg(e_{2,j}) =
\lcm(m_{s_1},m_{s_2})$, so that $[d_2]_{s,j} \in \mathbf{m}$ for all
$s,j$. By construction, all entries of $d_1$ are in $\mathbf{m}$ and
we can conclude that this resolution is minimal.
\end{proof}

 Next we show that all monomial ideals of projective dimension~$1$ (or
 their squarefree polarizations) can be characterized as
 $\mathcal{N}(\Delta^\vee)$ where $\Delta$ is a quasi-forest. This
 fact itself is known: Herzog, Hibi, and Zheng~\cite{Herzog-Dirac}
 proved it by using the Hilbert-Burch Theorem~\cite{Eisenbud1995}, and
 interpreting aspects of this theorem in the context of the
 Stanley-Reisner ring of the Alexander Dual of a quasi-tree.

Our proof, on the other hand, gives a specific and simple construction
of graph trees that support a resolution of
$\mathcal{N}(\Delta^\vee)$. The minimality of the resolution is
guaranteed by the previous lemma. 

\begin{theorem} \label{Quasi forest then resolution supported on a graph tree}
If $\Delta$ is a quasi-forest, then $S/\mathcal{N}(\Delta^\vee)$ has a minimal resolution which is supported on a tree.
\end{theorem}

   \begin{proof} First we shall construct a tree $T$ whose
   vertices will be labeled by the monomial generators of
   $\mathcal{N}(\Delta^\vee)$. Then we will show that the forest
   induced by the $\lcm$ of any two of the vertex labels is
   connected. If these induced forests are connected then so is any
   forest induced by an element of the $\lcm$-lattice of $I$ and the
   rest follows from Theorem 3.2 of~\cite{Faridi-MonRes}.

To construct the tree we do the following:
\begin{itemize}
\item[1)] Order the facets of $\Delta$ as $F_0,...,F_q$, so that $F_i$ is a leaf of $\Delta_i = \langle F_1,...,F_i \rangle$.
\item[2)] Start with the one vertex tree $T_0 = (V_0,E_0)$ where $V_0 = \{v_0\}$ and $E_0 = \emptyset$
\item[3)] For $i = 1,...,q$ do the following:
\begin{itemize}
\item[-] Pick $u < i$ such that $F_u$ is a joint of the leaf $F_i$ in $\Delta_i$
\item[-] Set $V_i = V_{i-1}\cup \{v_i\}$
\item[-] Set $E_i = E_{i-1}\cup \{(v_i,v_{u})\}$
\end{itemize}
\end{itemize}

What we get is a graph $T = (V_q,E_q)$ which, by construction, is a
tree. To complete our construction we determine a labeling of the
vertices of $T$ by which to homogenize. To do this we label the vertex
$v_i$ with the monomial
\begin{displaymath}
m_i = \prod_{x_j\in W \setminus F_i} x_j
\end{displaymath} 
where $W = \{x_1,..., x_n \}$ is the vertex set of $\Delta$. By
Lemma~\ref{Facet Generator Correspondence}, these labels are the
monomial generators of $\mathcal{N}(\Delta^\vee)$, so we have
constructed a tree and specified a labeling. The $I$-homogenization
of $T$ with respect to this labeling results in the $I$-complex
$\mathbf{F}_T$. We are left with proving that $\mathbf{F}_T$ is a
resolution.

Since $T$ is a tree, and hence a simplicial tree, to show that
$\mathbf{F}_T$ supports a resolution of $\mathcal{N}(\Delta^\vee)$ it
is sufficient to show that $T$ is connected on the subgraphs $T_{i,j}$
which are the induced subgraphs on the vertices $m_k$ such that $m_k
\big{|} \lcm(m_i,m_j)$, for any minimal generators $m_i, \ m_j$ in
$I$. We first observe that
\begin{displaymath}
\lcm(m_i,m_j) = \prod_{x_l\in W \setminus F_i\cap F_j} x_l
\end{displaymath}
so that
\begin{displaymath}
m_k \big{|}\lcm(m_i,m_j) \IFF F_i\cap F_j\subset F_k\
\end{displaymath}

Now, to show that every $T_{i,j}$ is connected we define $A_{i,j}$ to be the set 
\begin{displaymath}
A_{i,j} = \{0\leq k\leq n : m_k|\lcm(m_i,m_j)\} = \{0\leq k\leq n : F_i\cap F_j\subset F_k\}
\end{displaymath} 
and let $l$ be the smallest integer in $A_{i,j}$. We will show that for each $k\in A_{i,j}$, there is a path in $T_{i,j}$ connecting $v_k$ and $v_l$. \\

If $k\in A_{i,j},\ k\not = l$ then we can consider the facet $F_k$ in $\Delta_k$ which is a leaf, so it has a joint $F_{k_J}$ for some $k_J < k$. Since $l<k$, $F_l$ is a facet of $\Delta_k$ as well. This means that 
\begin{displaymath}
F_i\cap F_j \subset F_k\cap F_l \subset F_{k_J} \implies F_i\cap F_j \subset F_{k_J} \implies k_J \in A_{i,j}.
\end{displaymath}

 Since $k_J \in A_{i,j}$ for any joint of $F_k \in \Delta_k$, it is
 true for the specific joint we used in Step (3) of our construction
 of $T$. We may also conclude that $k_J \geq l$, by the minimality of
 $l$. Hence it is the case that the edge $\{ v_k,v_{k_J} \} \in T$
 which in turn implies that $\{ v_k,v_{k_J} \} \in T_{i,j}$. Since $l
 \leq k_J<k$, we can iterate this argument for $k_J$ and its joint in
 $\Delta_{k_J}$, and so on, finitely many times to get a path from
 $v_k$ to $v_l$ in $T_{i,j}$.
\end{proof}

\begin{remark} In the construction of $T$, we had some
  choice as to what joint we chose for a facet $F_k$ in the simplicial
  complex $\Delta_k$, hence the tree that we constructed is not
  unique. Furthermore, the proof follows through regardless of our
  choices, so that any tree that we may have constructed would give us
  a resolution of $\mathcal{N}(\Delta^\vee)$.
\end{remark}

\begin{example}
Let $\Delta$ be the simplicial tree

\begin{figure}[H] \centering
 \includegraphics{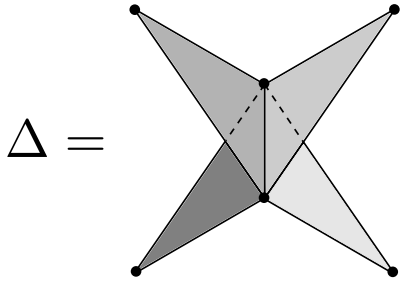}
 \caption{Quasi-tree with many leaf orders}
\end{figure}

Every order on the facets of $\Delta$ is a leaf order, every facet is a leaf, and every facet is the joint of every other facet. This means that if we use the construction given in the proof of Theorem \ref{Quasi forest then resolution supported on a graph tree}, we could produce any tree on four vertices. The minimal generators of $\mathcal{N}(\Delta^\vee)$ are $ x_1x_2x_3, \ x_1x_2x_4, \ x_1x_3x_4, \ x_2x_3x_4$ and the $\lcm$ of any two of these generators is $x_1x_2x_3x_4$, so that each $T_{i,j} = T$ for any tree $T$ we choose to consider. Hence, the $T_{i,j}$ are always connected and we will always get a minimal free resolution of $S/\mathcal{N}(\Delta^\vee)$.
\end{example}

\begin{remark} Fl{\o}ystad~\cite{Floystad} also
constructs specific trees supporting minimal resolutions for the class
of Cohen-Macaulay monomial ideals of projective dimension~1. Let
$I=(m_1,\ldots,m_q)$ be such an ideal and without loss of generality
we assume that $I$ is squarefree (otherwise replace generators with
their polarizations), and that the generators have been arranged so
that each $m_i$ corresponds to the complement of a facet $F_i$ of a
quasi-tree $\D$, and $F_1,\ldots,F_q$ is a leaf ordering of
$\D$. These extra arrangements are in place so that we can compare the
resulting graph with the one in Theorem~\ref{Quasi forest then
resolution supported on a graph tree}.

Consider the complete graph $\K$ on $q$ vertices, and label its
vertices with $m_1,\ldots,m_q$, and label each edge with the $\lcm$ of the
labels of its vertices. Starting at $i=1$, let $\K_i$  be the subgraph
of $\K$ consisting of all vertices and edges whose monomial labels
have total degree $\leq i$, and let $ \U_i$ be a spanning forest of
$\K_i$, with the condition that $\U_1 \subseteq \U_2 \subseteq
\ldots$. Let $d$ be the smallest integer for which $\U_d$ is connected
and contains all the vertices of $\K$. We use the notation $\T_G$ for
the tree $\U_d$. Fl{\o}ystad shows in~\cite{Floystad} that $\T_G$
supports a resolution of $I$.

We now show that a tree $\T_B$ obtained using the algorithm in
Theorem~\ref{Quasi forest then resolution supported on a graph tree}
is an instance of a $\T_G$ as described above. Suppose we have such a
tree $\T_B$, and consider for every $i$ its subgraph $(\T_B)_i$
consisting of edges and vertices whose monomial labels have total
degree $\leq i$. Then $(\T_B)_i$ is a spanning forest of $\K_i$, and
we have the chain of inclusions $(\T_B)_1 \subseteq (\T_B)_2 \subseteq
\ldots$.

Now suppose that the maximum degree of a vertex or edge label in
$\T_B$ is $d$ so that $\T_B= (\T_B)_d$. Then, if you drop the edges
and vertices with label of degree $d$, we have $(\T_B)_{d-1}\subsetneq
(\T_B)_d = \T_B$, which shows that $\T_B$ is an example of a $\T_G$.

\end{remark}

In order to prove a converse statement to Theorem \ref{Quasi forest then resolution supported on a graph tree}, we are going to need a couple of auxiliary results.

\begin{lemma} \label{Generators of Ideal on Induced Subcomplex}
Let $\Delta$ be a simplicial complex on $V = \{ x_1,...,x_n \}$, let $W = \{ x_1,...,x_t \} \subseteq V$, and let $\Delta_W$ be the induced subcomplex of $\Delta$ on $W$. If $m_1,...,m_r$ are the minimal generators of $\mathcal{N}(\Delta^\vee)$, then the generators of $\mathcal{N} \big( (\Delta_W)^\vee \big)$ are a subset of $\{ \gcd (m_1,x_1 \cdots x_t),...,\gcd (m_r,x_1 \cdots x_t) \}$
\end{lemma}
Before we begin it is worth noting that restricting to the first $t$ vertices is notationally convenient, but the statement will hold for any subset of $V$ (just make an appropriate relabeling of the vertices). Also, note that if we had considered $\Delta_W$ as a subcomplex of the $n$-simplex, the above statement would not hold.
\begin{proof}
If we present $\Delta$ as $\langle F_1,...,F_r \rangle$ then the generators of $\mathcal{N}(\Delta^\vee)$ have the form $m_i = \displaystyle{\prod_{x_j \in V \setminus F_i}} x_j$. We also know that the facets of $\Delta_W$ are subsets of the facets of $\Delta$, so we can present $\Delta_W$ as $\langle \overbar{F}_{i_1},..., \overbar{F}_{i_s} \rangle$, where $\{ i_1,...,i_s \} \subseteq \{1,...,r \}$ and $\overbar{F}_{i_j} \subseteq F_{i_j}$. Since $\overbar{F}_{i_j} = F_{i_j} \cap W$ we get that
\begin{displaymath}
W \setminus \overbar{F}_{i_j} = W \setminus (F_{i_j} \cap W) = (V \setminus F_{i_j}) \cap W
\end{displaymath}
and the generators of $\mathcal{N} \big( (\Delta_W)^\vee \big)$ are 
\begin{displaymath}
\overbar{m}_{i_j} = \prod_{\substack{x_s \not \in \overbar{F}_{i_j} \\ x_s \in W}} x_s = \prod_{\substack{x_s \in V \setminus F_{i_j} \\ x_s \in W}} x_s = \gcd (m_{i_j}, x_1 \cdots x_t)
\end{displaymath}
so $\overbar{m}_{i_j} \in \{ \gcd (m_1,x_1 \cdots x_t),...,\gcd (m_r,x_1 \cdots x_t) \}$.
\end{proof}
\begin{remark} \label{Redundant Generators}
In the above proof we used the fact that there is a correspondence between the facets of $\Delta_W$ and a subset of the facets of $\Delta$. If $F_q$ is a facet of $\Delta$ where $q \not \in \{i_1,...,i_s \}$ we still have that $F_q \cap W$ is a face of $\Delta_W$. Therefore, $F_q \cap W$ must be a subset of some facet $\overbar{F}_{i_j}$ of $\Delta_W$. With this information we can deduce that
\begin{displaymath}
\gcd (m_q, x_1 \cdots x_t) = \big( \gcd (m_{i_j}, x_1 \cdots x_t) \big) \prod_{\substack{x_s \in F_{i_j} \setminus F_q \\ x_s \in W}} x_s
\end{displaymath}
This tells us that $\gcd (m_q, x_1 \cdots x_t) \in \mathcal{N} \big( (\Delta_W)^\vee \big)$. What this allows us to do is say that
\begin{displaymath}
\mathcal{N} \big( (\Delta_W)^\vee \big) = \big( \gcd (m_1, x_1 \cdots x_t),..., \gcd (m_r, x_1 \cdots x_t) \big)
\end{displaymath}
With this fact we are able to prove the following corollary of Lemma~\ref{Generators of Ideal on Induced Subcomplex}.
\end{remark}
\begin{corollary}
Let $\Delta$ be a simplicial complex on $V = \{x_1,...,x_n\}$. Let $W = \{x_1,...,x_t\}$ for some $t \leq n$ and let $S' = k[x_1,...,x_t]$. Then
\begin{displaymath}
\dfrac{S'}{\mathcal{N} \big( (\Delta_W)^\vee \big)} \cong \dfrac{S}{\mathcal{N}(\Delta^\vee)} \otimes_S \dfrac{S}{( x_{t+1}-1,...,x_n-1)}
\end{displaymath}
\end{corollary}
\begin{proof}
Let $m_1,...,m_r$ be the minimal generators for $\mathcal{N}(\Delta^\vee)$. Remark \ref{Redundant Generators} tells us that
\begin{displaymath}
\mathcal{N} \big( (\Delta_W)^\vee \big) = \big( \gcd (m_1, x_1 \cdots x_t),..., \gcd (m_r, x_1 \cdots x_t) \big)
\end{displaymath}
which is the same as saying that we can form the generators of $\mathcal{N} \big( (\Delta_W)^\vee \big)$ by taking the the generators of $\mathcal{N}(\Delta^\vee)$ and setting the variables $x_{t+1},...,x_n$ equal to 1. When we are using quotient modules we can do this by adding the desired relations to the ideal by which we are taking the quotient. Specifically, what we mean is
\begin{displaymath}
\dfrac{S'}{\mathcal{N} \big( (\Delta_W)^\vee \big)} \cong \dfrac{S}{\mathcal{N}(\Delta^\vee) + \big( x_{t+1}-1,..., x_n-1 \big)}
\end{displaymath}
Moreover, we have that
\begin{displaymath}
\dfrac{S}{\mathcal{N}(\Delta^\vee) + (x_{t+1}-1,...,x_n-1)} \cong \dfrac{S}{\mathcal{N}(\Delta^\vee)} \otimes_S \dfrac{S}{(x_{t+1}-1,...,x_n-1)}
\end{displaymath}
which is the desired result.

\end{proof}

With these additional results we are now able to provide a new proof the following theorem. 

\begin{theorem}[Herzog, Hibi, Zheng~\cite{Herzog-Dirac}] \label{HHZ}
  Let $\Delta$ be a simplicial complex which is not a simplex, then $\pd(\mathcal{N}(\Delta^\vee)) = 1$ if and only if $\Delta$ is a quasi-forest.
\end{theorem}

\begin{proof}
($\Leftarrow$) Follows from Theorem~\ref{Quasi forest then resolution
    supported on a graph tree}.
  
($\Rightarrow$) Without loss of generality let $W = \{x_1,..., x_k\}$. Recalling Proposition~\ref{prop 2}, it is enough to show that $\Delta_W$ has a leaf to conclude that $\Delta$ is a quasi-forest. Let $\mathbf{F}$ be the minimal free resolution 
\begin{equation*}
\begin{xy}
(-5,0)*+{0}="P0";
(10,0)*+{S^{r-1}}="P1";
(25,0)*+{S^r}="P2";
(40,0)*+{S}="P3";
(55,0)*+{0}="P4";
{\ar "P2"; "P3"}%
{\ar "P0"; "P1"}%
{\ar "P1"; "P2"}%
{\ar "P3"; "P4"}%
\end{xy}
\end{equation*}

of $S/\mathcal{N}(\Delta^\vee)$. The elements $x_{t+1}-1,...,x_n-1$ form an $S/\mathcal{N}(\Delta^\vee)$-sequence, so we can construct the resolution
\begin{displaymath}
\mathbf{F} \otimes_S \dfrac{S}{(x_{t+1}-1,...,x_n-1)}
\end{displaymath}
of $S'/\mathcal{N} \big( (\Delta_W)^\vee \big)$, where $S' = k[x_1,...,x_t]$ (See Chapters 20 and 21 of~\cite{Peeva-GradedSyzygies} for further details). Since the length of the resulting resolution is no greater than the length of $\mathbf{F}$, we find that $\mathrm{pd} (\mathcal{N} \big( (\Delta_W)^\vee \big)) \leq \mathrm{pd} (\mathcal{N}(\Delta^\vee)) = 1$.

If $\mathrm{pd} (\mathcal{N} \big( (\Delta_W)^\vee \big)) = 0$ then it must be the case that $\mathcal{N} \big( (\Delta_W)^\vee \big)$ is principal, in which case $\Delta_W$ is a simplex and therefore has a leaf. If $\mathrm{pd} (\mathcal{N} \big( (\Delta_W)^\vee \big)) = 1$, then Theorem \ref{pd is one iff graph tree} tells us that $\mathcal{N} \big( (\Delta_W)^\vee \big)$ has a minimal resolution supported on a tree $T$. Choose a labeling of the vertices of $T$ for which the $\mathcal{N} \big( (\Delta_W)^\vee \big)$-homogenization yields a resolution, let $\overbar{m}_l$ be the label of one of the free vertices of $T$ and let $\overbar{m}_j$ be the label of the vertex which shares an edge with $m_l$. For any other minimal generator $\overbar{m}_i$ of $\mathcal{N} \big( (\Delta_W)^\vee \big)$ we must have that $\overbar{m}_j \big{|} \lcm (\overbar{m}_l,\overbar{m}_i)$ to ensure connectivity of the induced forest generated by the $lcm$ of $\overbar{m}_l$ and $\overbar{m}_i$. In the proof of Theorem \ref{Quasi forest then resolution supported on a graph tree} we saw that
\begin{displaymath}
\overbar{m}_j \big{|}\lcm(\overbar{m}_l,\overbar{m}_i) \IFF
\overbar{F}_l\cap \overbar{F}_i \subset \overbar{F}_j
\end{displaymath} 
which is exactly the condition needed for $\overbar{F}_l$ to be a leaf of $\Delta_W$ with joint $\overbar{F}_j$. Hence, we can conclude that $\Delta$ is a quasi-forest.
\end{proof}

\begin{example}
  Let $I = (\overbrace{x_1x_3x_6}^{m_1}, \overbrace{x_1x_4x_6}^{m_2},  \overbrace{x_1x_2x_4}^{m_3}, \overbrace{x_4x_5x_6}^{m_4}) \subset k[x_1,...,x_6]$ be the ideal from example \ref{Example of Homogenization}. Using lemma \ref{Facet Generator Correspondence} we get that $\mathcal{N}(I^\vee)$ is the simplicial complex \\

\begin{figure}[H] \centering
 \includegraphics{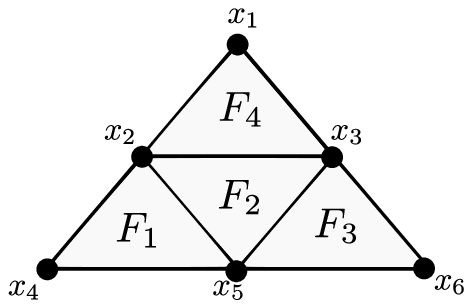}
\end{figure}
\noindent where the facet $F_i$ corresponds to the monomial $m_i$, for $i = 1,...,4$. We see that $\mathcal{N}(I^\vee)$ is a quasi-tree with leaf order $F_1,F_2,F_3,F_4$. By Theorem \ref{Quasi forest then resolution supported on a graph tree} we have that $\pd(\mathcal{N}((\mathcal{N}(I^\vee))^\vee)) = 1$ and remark \ref{Complex-Ideal Correspondence} tells us that $\mathcal{N}((\mathcal{N}(I^\vee))^\vee) = I$, therefore we may use the algorithm presented in Theorem \ref{Quasi forest then resolution supported on a graph tree} to construct a tree which supports the minimal free resolution of $S/I$.

\begin{figure}[H] \centering
 \includegraphics{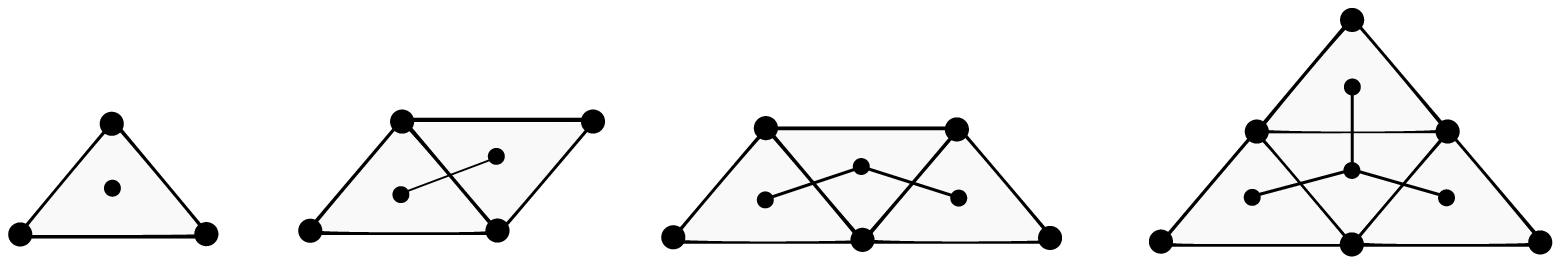}
\end{figure}

\noindent Hence the simplicial complex supporting the minimal free resolution of $S/I$ is given by the labeled tree

\begin{figure}[H] \centering
 \includegraphics{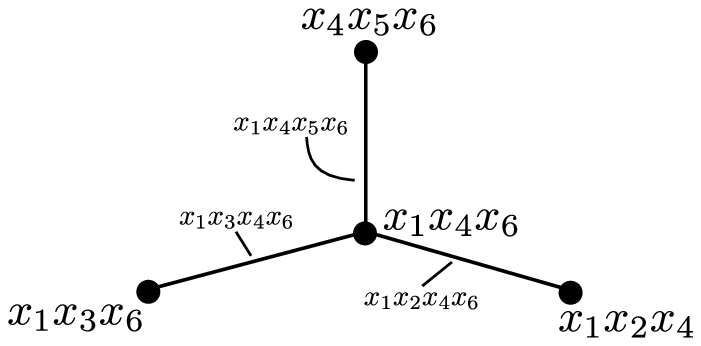}
\end{figure}
\end{example} 
By using the correspondence given by $\mathcal{N}((-)^\vee)$ we are able to identify precisely when a squarefree monomial ideal $I$ is such that $\pd(I) \leq 1$, and when this is the case we have also provided an algorithm for constructing the minimal free resolution of $S/I$ from the simplicial complex $\mathcal{N}(I^\vee)$. The amalgamation of these results and techniques is the following theorem, in which classify monomial ideals with projective dimension $\leq 1$.

\begin{theorem}\label{Cor to HHZ} Let $I$ be a squarefree monomial ideal in a polynomial ring $S$. Then the following are equivalent.
\begin{enumerate}  
\item $\pd_S(I) \leq 1$
\item $\mathcal{N}(I^\vee)$ is a quasi-forest
\item $S/I$ has a minimal free resolution supported on a graph-tree.
\end{enumerate}
\end{theorem}
%


\end{document}